\documentclass{amsart}

\usepackage{amsthm}
\usepackage{graphicx}

\title{Variance bounds for disc-polygons}

\author[F. Fodor]{Ferenc Fodor}
\address{Department of Geometry, Bolyai Institute, University of Szeged, Aradi v\'ertan\'uk tere 1, 6720 Szeged, Hungary}
\email{fodorf@math.u-szeged.hu}

\author[B. Gr\"unfelder]{Bal\'azs Gr\"unfelder}
\address{University of Szeged, Aradi v\'ertan\'uk tere 1, 6720 Szeged, Hungary}
\email{grunfelder.balazs@stud.u-szeged.hu}

\author[V. V\'igh]{Viktor V\'\i gh}
\address{Department of Geometry, Bolyai Institute, University of Szeged, Aradi v\'ertan\'uk tere 1, 6720 Szeged, Hungary}
\email{vigvik@math.u-szeged.hu}

\usepackage[nobysame,initials]{amsrefs}
\usepackage{bbm}

\usepackage{hyperref}
\usepackage{xcolor}
\usepackage[normalem]{ulem}

\newcommand{\R}{\mathbb{R}}
\newcommand{\N}{\mathbb{N}}
\newcommand{\Ex}{\mathbb{E}}
\newcommand{\F}{\mathcal{F}}

\newtheorem{theorem}{Theorem}
\newtheorem{lemma}{Lemma}

\newcommand{\conv}{{\rm conv }}

\newcommand{\pr}{\mathbb P}
\newcommand{\indi}{\mathbbm{1}}

\DeclareMathOperator{\Var}{Var}
\DeclareMathOperator{\inti}{int}

\DeclareMathOperator{\Vol}{Vol}

\allowdisplaybreaks

\date{\today}

\begin{document}

\begin{abstract}
	We prove asymptotic lower bounds on the variance
	of the number of vertices and missed area of random disc-polygons in convex discs whose boundary is $C_+^2$ smooth. The established lower bounds are of the same order as the upper bounds proved previously in \cite{FV18}.
\end{abstract}


\maketitle
\bibliographystyle{amsplain}

\section{Results}
We work in the Euclidean plane $\R^2$ with origin centred closed unit ball $B=B^2$ whose boundary is the unit circle $S^1=\partial B$. We denote the (origin centred) open unit ball by $B^\circ$. We denote by $A(\cdot)$ the area of measurable sets in $\R^2$. For general information about convex sets we refer to the books by Gruber \cite{G07} and Schneider \cite{Sch14}. 

For asymptotic inequalities, we use the following common notation: for two real sequences $f,g$, we write $f\ll g$ if there is a positive constant $\gamma$ such that $|f(n)|\leq\gamma g(n)$ for every $n\in\N$.

Let $K\subset\R^2$ be a convex disc with $C^2_+$ smooth boundary (twice continuously differentiable with positive curvature $\kappa(x)>0$ at every point $x\in\partial K$). Let $r_M=1/\kappa_m$, where $\kappa_m=\min_{x\in\partial K} \kappa (x)>0$. It is known (see \cite[Theorem~3.2.12 on p. 164]{Sch14}) that $K$ slides freely in a circle of radius $r_M$, that is, for every $x\in\partial K$, there exists a $v\in\R^2$ with $x\in r_MS^1+v$ and $K\subset r_MS^1+v$ (cf. \cite[p. 156]{Sch14}). For $r\geq r_M$ and $X\subset K$, let $\conv_r (X)$ denote the intersection of all radius $r$ closed circular discs that contain $X$, that is,
$$\conv_r(X):=\bigcap_{\overset{v\in\R^2,}{X\subset rB+v}} (rB+v).$$
The set $\conv_r(X)$ is called the closed $r$-spindle convex hull or $r$-hull of $X$. It is known that $\conv_r(X)\subset K$, see Bezdek et al \cite{BLNP05}. For more information about the geometric properties of the $r$-spindle convex hull see, for example, \cite{BLNP05} and \cite{FKV14} and the references therein.

We investigate the following probability model. Let $X_n=\{x_1, \ldots, x_n\}\subset K$ be a sample of $n$  i.i.d. random points selected according to the uniform probability distribution. Let $K_n^r:=\conv_r(X_n)$.
The set $K_n^r$ is a {\em (uniform) random convex $r$-disc-polygon in $K$} whose sides are arcs of radius $r$ circles. The edges, vertices and angles of $K_n^r$ are defined the usual way.  Let $f_0(K_n^r)$ denote the number of vertices of $K_n^r$. We call $A(K\setminus K_n^r)$ the missed area of $K_n^r$. 

If $\overline{K}_n$ denotes the (usual) convex hull of $X_n$, then $\overline{K}_n\subset K_n^r\subset K$ for all $r\geq r_M$. Since the containment $\overline{K}_n\subset K_n^r$ is strict, $K_n^r$ approximates $K$ better than $K_n$ from the point of view of area and perimeter. It is also clear that for fixed $X_n$, the $r$-disc-polygons $K_n^r$ tend to $\overline{K}_n$ in the Hausdorff-metric as $r\to\infty$. Furthermore, $f_0(K_n^r)\leq f_0(\overline{K}_n)$ for all $r\geq r_M$. 

The geometric properties of the (classical) random polygon $\overline{K}_n$ have been investigated extensively. Starting with the seminal papers of R\'enyi and Sulanke \cites{RS63, RS64, RS68} the asymptotic behaviour of the expected number of vertices, area and perimeter have been determined. For a detailed overview of known results about the classical model, see for example, the surveys \cites{BBS06, B08, R10} and \cite{Sch17}.

Fodor, Kevei and V\'igh \cite[Theorem~1.1 on p. 901]{FKV14} proved the following asymptotic formulas: for $r>r_M$, it holds that
\begin{align}
\lim_{n\to\infty} \Ex f_0(K_n^r)\cdot n^{-1/3}&=\sqrt[3]{\frac{2}{3A(K)}}\Gamma\left (\frac 53\right )c_1(K,r),\label{vertexno}\\
\lim_{n\to\infty} \Ex A(K\setminus K_n^r)\cdot n^{2/3}&=\sqrt[3]{\frac{2A^2(K)}{3}}\Gamma\left (\frac 53\right )c_1(K,r),\label{area}
\end{align}
where the constant
$$c_1(K,r)=\int_{\partial K}\left (\kappa(x)-\frac 1r\right )^{1/3}\, dx,$$
seems to resemble to the affine arc-length although it is unclear what is its exact geometric meaning. The formulas \eqref{vertexno} and \eqref{area} are generalizations of the corresponding results of R\'enyi and Sulanke for the classical case, see Section~3 of \cite{FKV14}. We also note that in \eqref{vertexno} and \eqref{area} the condition that $r>r_M$ is important. If $K$ is a disc-polygon of radius $r$ itself, then the order of magnitude of $\Ex f_0(K_n^r)$ and $\Ex A(K\setminus K_n^r)$ are different, see \cite{FKV21+}.


In the case when $K=B$, it is proved in \cite[Theorem~1.3 on p. 902]{FKV14} that for $r=1$, it holds that
\begin{align*}
\lim_{n\to\infty} \Ex f_0(B^1_n)&=\frac{\pi^2}{2},\\
\lim_{n\to\infty} \Ex A(B\setminus B^1_n)\cdot n&=\frac{\pi^3}{2}.
\end{align*}

Significantly less is known about the higher moments of these quantities. 

In the classical case,  Reitzner \cite{R03, R05a} proved the following upper bounds for the variance of the number of $j$-dimensional faces and the volume of the random polytopes in smooth convex bodies in dimension $d$:  
\begin{align*}
	\Var(\Vol (\overline{K}_n))&\ll n^{-(d+3)/(d+1)},\\
	\Var(f_j(\overline{K}_n))&\ll n^{(d-1)/(d+1)}.
\end{align*}
Here  $\overline{K}_n$ denotes the (classical) convex hull of $n$ i.i.d. random points selected from the $d$-dimensional convex body $K\subset\R^d$ with $C^2_+$ smooth boundary. 
The symbol $\Vol(\cdot)$ denotes the volume of Lebesgue measurable sets in $\R^d$ and $f_j(\cdot)$ is the number of $j$-dimensional faces. The implied constants depend only on $K$ and the dimension. The upper bounds also imply strong laws of large numbers for these quantities. 

Reitzner \cite{R05b} gave matching lower bounds for the variance in case $K$ is smooth:
\begin{align*}
	\Var(\Vol(\overline{K}_n))&\gg n^{-(d+3)/(d+1)},\\
	\Var(f_j(\overline{K}_n))&\gg n^{(d-1)/(d+1)}. 
\end{align*}
Using these lower bounds Reitzner \cite{R05b} established central limit theorems for the number of $j$-dimensional faces and the volume for smooth convex bodies. 

Upper and lower bounds, laws of large numbers and central limit theorems were extended for the case when $K\subset \R^d$ is a polytope by B\'ar\'any and Reitzner \cite{BR10}. B\'ar\'any and Steiger \cite{BS13} proved asymptotic upper bounds and strong laws for the missed area and the number of vertices in the classical random model for arbitrary convex discs in $\R^2$ without any smoothness condition.  

Fodor and V\'\i gh \cite{FV18} proved asymptotic upper bounds for the variance of the vertex number and the missed area of uniform random disc-polygons in $C^2_+$ smooth convex discs: For any $r>r_M$ it holds that
\begin{align}
	\Var (f_0(K_n^r))&\ll n^{1/3}\label{disc-vertex-upper},\\
	\Var(A(K_n^r))&\ll n^{-5/3}\label{disc-area-upper},
\end{align}	
where the implied constants depend only on $K$ and $r$. 
In the case when $K=B^2$, they proved \cite{FV18} that
\begin{align}
	\Var (f_0(K_n^r))&\approx 1,\label{circle}\\
	\Var(A(K_n^r))&\ll n^{-2}.
\end{align}
where the implied constants are universal.
Note that the lower bound in \eqref{circle} follows from the fact that the expected number of vertices $\Ex f_0(K_n)$ is a non-integer constant. Formulas \eqref{disc-vertex-upper}, \eqref{disc-area-upper} and \eqref{circle} imply laws of large number for the corresponding quantities, see \cite{FV18}.  

In this paper we prove matching lower bounds for the variance of the vertex number and the missed area for the case when $r>r_M$. Our main results are summarized in the following theorem.

\begin{theorem}\label{thm:also becsles}
Let $K$ be a convex disc whose boundary is 
of class $C^2_+$. For any $r>r_M$ it holds that
\begin{align}
\label{vertex-variance-a}
\Var (f_0(K_n^r))&\gg n^{\frac 13},\\
\label{area-variance-a}
\Var (A(K_n^r))&\gg n^{-\frac 53},
\end{align}
where the implied constants depend only on $K$ and $r$.  
\end{theorem}

We note that Theorem~\ref{thm:also becsles} is contained in the Master thesis of B. Gr\"unfelder \cite{G21}, and also in his Hungarian OTDK student competition paper.  

We achieve these results using a modified version of the method of Reitzner \cite{R05a} which had already been used in adapted forms in several different settings, see, for example, \cite{BFRV09, TW18, TTW18}. In the disc-polygonal setting, the difficulty lies in the intrinsic geometry of the model.

The lower bounds in Theorem~\ref{thm:also becsles} may open the road towards quantitative central limit theorems, similarly as in \cites{TTW18, TW18, BRT21+} using normal approximation bounds from Stein's method.

The layout of the paper is the following: in Section~\ref{preparations} we collect some necessary preparatory material. Section~\ref{proof-area} contains the proof of \eqref{area-variance-a} and Section~\ref{proof-vertex} contains the outlines of the changes necessary for the proof of \eqref{vertex-variance-a}. 

\section{Preparations}\label{preparations}
Without loss of generality, we may assume that $r=1$ and prove Theorem~\ref{thm:also becsles} for the case when $r_M<1$, since the general statements follow by a scaling argument.
For simplicity, we write $K_n$ for $K_n^1$.

For $p\in\R^2$, the set $K\setminus (B^\circ+p)$ is called a {\em disc-cap (of radius $1$)} of $K$.
We use the notations from \cite{FKV14}.
Let $x$ and $y$ be two points in $K$. 
The two unit circles that pass through these points, determine two disc-caps of $K$, denoted by $D_-(x,y)$ and $D_+(x,y)$, respectively, such that $A(D_-(x,y))\leq A(D_+(x,y))$.
Briefly, we write $A_-(x,y)=A(D_-(x,y))$ and $A_+(x,y)=A(D_+(x,y))$.
Lemma~4.3 in \cite{FKV14} states that if the boundary of $K$ is of class $C^2_+$ ($r_M<1$), then there exists a $\delta>0$  (depending only on $K$) with the property that for any $x, y\in\inti K$ it holds that $A_+(x,y)>\delta$.

We need some further technical statements about general disc-caps. 
Denote the (unique) outer unit normal to $K$ at the boundary point $x$ by $u_x\in S^1$, and the unique boundary point with outer unit normal $u\in S^1$ by $x_u\in \partial K$. It is proved in Lemma~4.1 of \cite{FKV14} that if $K$ is a convex disc with $C^2_{{+}}$ boundary and $\kappa_m>1$, and
$D=K\setminus (B^\circ+p)$ is a non-empty disc-cap, then there exists a unique point $x_0\in \partial K \cap \partial D$ such that $p=x_0-(1+t) u_{x_0}$ for some $t\geq 0$. The point $x_0$ is the {\em vertex} and the number $t$ is the {\em height of $D$}.


Let us denote the disc-cap with vertex $x_u\in\partial K$ and height $t$ by $D(u,t)$.
To simplify the notation, we write $A(u,t)=A(D(u,t))$, and let $\ell(u,t)$ denote the arc-length of $\partial D(u,t)\cap (\partial B+x_u-(1+t)u)$. 
The latter also exists since for each $u\in S^1$, there exists a maximal positive constant $t^*(u)$ such that $(B+x_u-(1+t)u)\cap K\neq\emptyset$ for all $t\in [0, t^*(u)]$.

The following limit relations for $A(u,t)$ and $\ell(u,t)$ are proved (in a more precise form) in \cite[p. 905, Lemma 4.2]{FKV14}:
\begin{equation}\label{Vandellestimate2}
	\ell(u_x,t) \approx t^{1/2}, \quad\quad
	A(u_x,t)\approx t^{3/2},
\end{equation}
as $t\to 0^+$, where the implied constants depend only on $K$.

Let $D$ be a disc-cap of $K$ with vertex $x$. 
For a line $e\subset\R^2$ perpendicular to $u_x$, let $e_+$ denote the closed half plane that contains $x$. 
Then there exists a maximal cap $C_-(D)=K\cap e_+$ that is fully contained in $D$, and a minimal cap $C_+(D)=e'_+\cap K$ containing $D$.
We recall \cite[Claim~1 on p. 1146]{FV18}, that gives a relation between classical caps and disc-caps, as follows: There exists a constant $\hat c$ depending only on $K$ such that if the height of the disc-cap $D$ is sufficiently small, then 
\begin{equation}\label{sapkatart}
	C_-(D)-x\supset \hat c (C_+(D)-x).
\end{equation}
The relation \eqref{sapkatart} means that the area of a disc-cap can be bounded by two classical caps such that one of them is an enlarged image of the other one by a constant.

Let $x_i,x_j$ ($i\neq j$) be two points of $X_n$, and let $B(x_i,x_j)$ be one of the unit discs containing $x_i$ and $x_j$ on its boundary. 
The shorter arc of $\partial B(x_i,x_j)$ forms an edge of $K_n$ if the entire set $X_n$ is contained in $B(x_i,x_j)$. 
It may happen that the pair $x_i,x_j$ determines two edges of $K_n$ if the above condition holds for both unit discs that contain $x_i$ and $x_j$ on its boundary.

\section{Proof of \eqref{area-variance-a} in Theorem~3}\label{proof-area}
The proof is based on the ideas of Reitzner \cite{R05a}: we give small (disc-)caps which contribute to the variance geometrically independently and show that the variance in these caps is already sufficiently large.

For every $x\in\partial K$ and $t\in(0,1)$ consider the disc-cap $D(x,t)$ of vertex $x$ and height $t$.
Let the Euclidean cap of vertex $x$ and height $t$ be $C(x,t)$.
Let the line cutting off the cap $C(x,t)$ be $H(x,t)$.
Clearly, $D(x,t)\supset C(x,t)$.

In the following we use large values of $n$, thus by an inequality of type $\ll$ we always assume that $t$ is sufficiently small.

Denote the intersections of $\partial K$ and the line $H(x,t)$ by $w_1$ and $w_2$, and let $w_0=x$.
For the triangle $\Delta=\left[w_0,w_1,w_2\right]$ we have
\begin{equation*}
    \Delta\subset C(x,t)\subset D(x,t).
\end{equation*}

Let us define for $j=0,1,2$ the small triangles
\begin{equation*}
    \Delta_j=\Delta_j(x,t)=w_j+\frac1{20}(\left[w_0,w_1,w_2\right]-w_j),
\end{equation*}
i.e. we shrink the $\Delta$ from each of its vertices by a factor of $1/20$.
It follows from \eqref{sapkatart} that $A(\Delta_j(x,t))\approx t^{\frac32}$, since the order of magnitude of the height of triangle $\Delta$ is $t$ and of its base is $\sqrt{t}$.

Let $x$, $t$ and the points $z_1\in \Delta_1(x,t)$ and $z_2\in \Delta_2(x,t)$ be fixed. For $z_0\in \Delta _0(x,t)$ let $\hat{A}(z_0)$ denote the area of the non-convex triangular region $\widetilde{\Delta}(z_0)$ we obtain by joining $z_0$ with $z_1$ and $z_2$ by circular arcs of radius $1$ that are outside of the triangle $z_0z_1z_2$, and also joining $z_1$ and $z_2$ such that the arc intersects the interior of $z_0z_1z_2$.


  \begin{figure}[h]
	\centering
	\includegraphics[scale=.15]{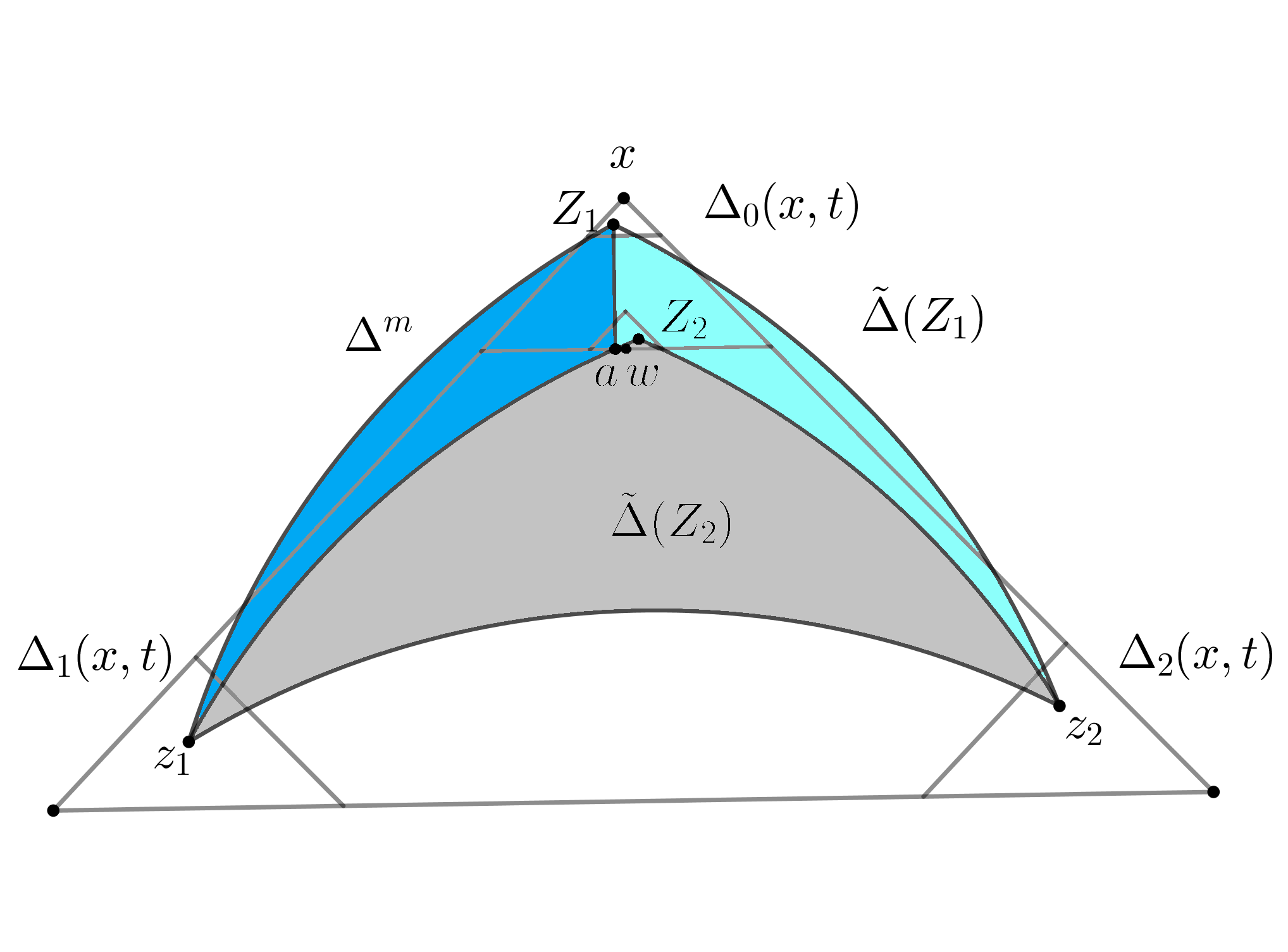}
	\caption{Splitting $\Delta^m$}\label{abra:kulonbseg}
\end{figure}

\begin{lemma}\label{sapka-szorasok}
    Let $Z$ be a uniform random point in $\Delta_0(x,t)$. Then
    \begin{equation*}
        \Var(\hat{A}(Z))\gg t^3.
    \end{equation*}
\end{lemma}

\begin{proof}
  
Let $w$ denote the midpoint of the side opposite to $x$ in the triangle $\Delta_0(x,t)$. Let 
$$\Delta_0^{(1)}(x,t)=x+\frac13(\Delta_0(x,t)-x)$$
and
$$\Delta_0^{(2)}(x,t)=w+\frac13(\Delta_0(x,t)-w),$$
i.e. in the triangle $\Delta_0(x,t)$ we take two smaller triangles which are the shrunk images of $\Delta_0(x,t)$ by a factor of $1/3$ from $x$ and $w$.
The area of $\Delta_0^{(1)}$ and $\Delta_0^{(2)}$ is one-ninth of that of $\Delta_0$, respectively.

For every $Z_1\in\Delta_0^{(1)}$ and $Z_2\in\Delta_0^{(2)}$, it holds that $\widetilde{\Delta}(Z_1)\supset \widetilde{\Delta}(Z_2)$, therefore $\hat{A}(Z_1)>\hat{A}(Z_2)$. Let $\Delta^m=\widetilde{\Delta}(Z_1)\setminus \widetilde{\Delta}(Z_2)$. We need $A(\Delta^m)$.

Cut $\Delta^m$ by a segment trough $Z_1$ perpendicular to the line $H$, and denote the other intersection point of this segment with $\partial\Delta^m$ by $a$.
Then $d(Z_1,a)\approx t$.
Suppose that after the cut $Z_2$ is contained in the set that has $z_2$ on its boundary, see Figure \ref{abra:kulonbseg}.

Consider the Euclidean triangle $[Z_1,a,z_1]$, and let $\gamma$ denote the angle at the vertex $z_1$.
The radius of the circumscribed circle of this triangle is of order $\sqrt{t}$.
Thus, since the side opposite to the angle $\gamma$ has is of order $t$, the law of sines gives  that $\sin{\gamma}\approx\sqrt{t}$.
By the smallness of $t$, the angle $\gamma$ has the same order of magnitude as $\sin{\gamma}$.

After that, we translate $a$ along $\partial\Delta^m$ into a point $a'$ such that $d(z_1,a')=d(z_1,Z_1)$ holds.
(In case $a'$ has reached $Z_2$ and the distances are still not equal, we translate $Z_1$ closer to $z_1$.)
By this, for the angle at $z_1$ we have $\gamma'\geq\gamma$.

Since $d(z_1,a')=d(z_1,Z_1)$, the angle between the two circular arcs of radius $1$ corresponding to the two segments is $\gamma'$ as well, see Figure \ref{abra:korivek}.
By the angle between two circular arc we mean the angle between their tangent lines.

\begin{figure}
	\centering
	\includegraphics[scale=.75]{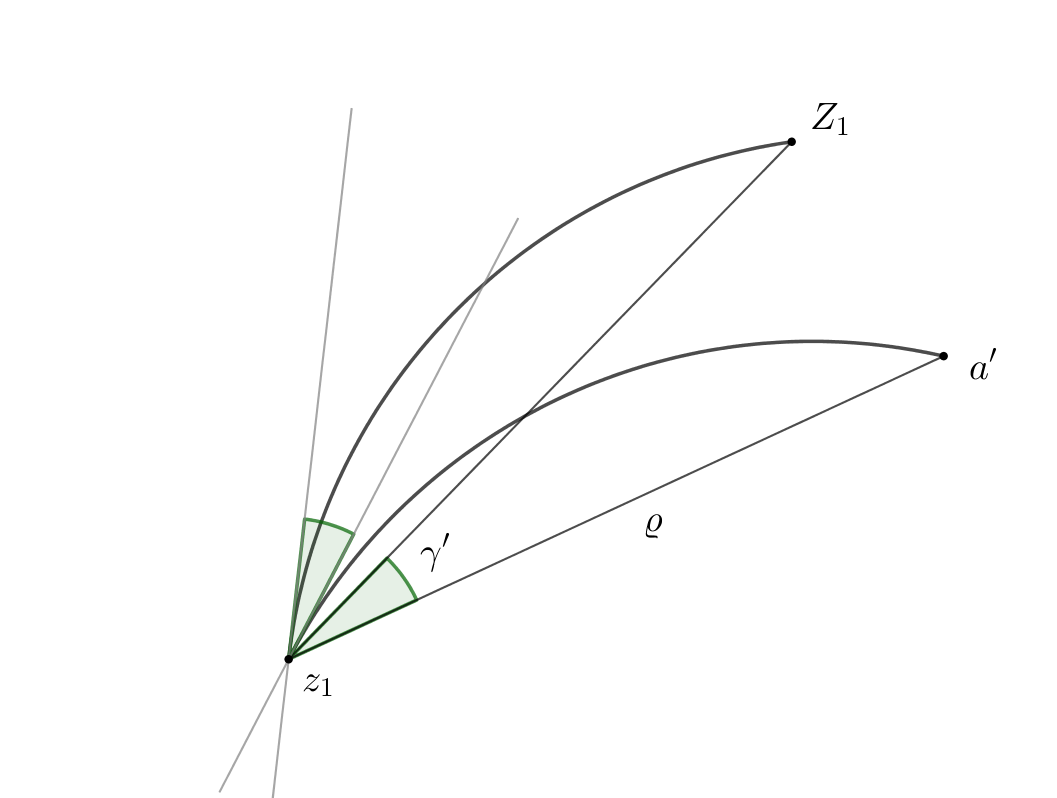}
	\caption{Angle of circular arcs and chords}	\label{abra:korivek}
\end{figure}

Consider the sector-like shape determined by $z_1,a'$ and $Z_1$, whose legs are circular arcs of radius $1$.
This is a part of a circular disc of radius $\varrho=d(z_1,Z_1)$, by which, rotating around the centre, we can cover the whole disc.
The area of a shape of this property is proportional to the central angle and the square of the radius.
Here we have $\varrho\approx\sqrt{t}$ and the order of magnitude of the angle is at least $\sqrt{t}$, therefore the area of this shape is at least of order $t^{3/2}$.

We have estimated the area $\hat{A}(Z_1)-\hat{A}(Z_2)$ we are looking for from below, examining its subset.
This gives us the following lower estimate:
\begin{equation}\label{ineq}
    \hat{A}(Z_1)-\hat{A}(Z_2)\gg t^{3/2}.
\end{equation}

Let $Z$, $Z'_1$ and $Z'_2$ be i.i.d uniform random points in $\Delta_0(x,t)$.
Using \eqref{ineq}, we obtain the desired lower bound:
\begin{align*}
    \Var(\hat{A}(Z))&=
    \Ex\left[\hat{A}(Z)^2\right]-\Ex\left[\hat{A}(Z)\right]^2=\\
    &=\Ex\left[\hat{A}(Z)^2\right]-\Ex\left[\hat{A}(Z'_1)\hat{A}(Z'_2)\right]=\\
    &=\frac12\Ex\left[\hat{A}(Z'_1)^2-2\hat{A}(Z'_1)\hat{A}(Z_2)+\hat{A}(Z'_2)^2\right]=\\
    &=\frac12\Ex\left[(\hat{A}(Z'_1)-\hat{A}(Z'_2))^2\right]\geq\\
    &\geq\frac12\Ex\left[\left(\hat{A}(Z'_1)-\hat{A}(Z'_2)\right)^2\indi(Z'_1\in\Delta_1,Z'_2\in\Delta_2)\right]\gg\\ 
    &\gg t^3 \Ex\left[\indi(Z'_1\in\Delta_1,Z'_2\in\Delta_2)\right]\gg t^3.
\end{align*} 
\end{proof}

We may assume that $n>n_0$ for some suitable $n_0$, since it is sufficient to prove the lower bound of variance for large $n$.
In the following, we consider disc-caps of height $t_n$ for 
\begin{equation}\label{tn-ek}
    t_n=n^{-\frac23}.
\end{equation}
Choose a maximal set of points $y_1,\ldots, y_m$ on $\partial K$ such that $|x_i-x_j|\geq 2\sqrt{c_2}\sqrt{t_n}$ for any $i,j\in \{1,\ldots, m\}$ for some constant $c_2$ that we specify later. 
Then 
\begin{equation}\label{m}
    m\gg n^{\frac13}.
\end{equation}

Consider the pairwise disjoint disc-caps $D(y_j,t_n)$ and the previously defined triangles $\Delta(y_j,t_n)$ in them, for $j\in\left[m\right]$.
For each $\Delta(y_j,t_n)$, also construct the small triangles $\Delta_i(y_j,t_n)$, for $i=0,1,2$.
Let $E_j$ be the event that each of the small triangles $\Delta_i(y_j,t_n)$ contains exactly one of the random points $x_1,\dots,x_n$ and that $D(y_j,c_2t_n)$ contains no other random point.
By \eqref{Vandellestimate2} we have
$$A(D(y_j,c_2t_n))\ll\frac1n,$$
and for $i=0,1,2$
$$A(\Delta_i(y_j,t_n))\gg\frac1n.$$

We have for every	$j\in\left[m\right]$
\begin{equation*}\label{valseg}
	\pr(E_j)\gg\binom{n}{3}\left(\frac1n\right)^3\left(1-\frac1n\right)^{n-3}\gg1,
\end{equation*}
thus
\begin{equation}\label{expectation}
	\Ex\left(\sum_{j=1}^{m}\indi(E_j)\right)=\sum_{j=1}^{m}\pr(E_j)\gg m.
\end{equation}

In the case the event $E_j$ occurs, let the random point in $\Delta_0(y_j,t_n)$ be denoted by $Z_j$.

\begin{lemma}\label{fuggetlen}
Assume that $J\subset\left[m\right]$ and $E_j$ occurs for every $j\in J$. Then

\begin{equation*}
    A(K_n)=A\left(\conv_r\left(X_n\setminus\{Z_1,\dots Z_j\}\right)\right)+\sum_{j\in J}\hat{A}(Z_j).
\end{equation*}
	\end{lemma}

\begin{proof}
	
Our goal is to show that if for indices	$j,k\in J$, $j\neq k$, $Z_j$ and $Z_k$ are the random points in triangles $\Delta_0(y_j,t_n)$ and $\Delta_0(y_k,t_n)$, then $Z_j$ and $Z_k$ are vertices of $K_n$ and there is no edge between them. This means that the contributions of $Z_i$ and $Z_j$ to the area of $K_n$ are geometrically "independent". 

For this, we need that for every $j$, the circular arc of radius $1$ determined by the two random points in triangles $\Delta_0(y_j,t_n)$ and $\Delta_i(y_j,t_n)$ ($i\in\{1,2\}$), meets the boundary of $K$ without intersecting any other disc-cap $D(y_k,t_n)$.

For simplicity take a fixed disc-cap of height $t$ and vertex $x$, and the triangles $\Delta_0$ and $\Delta_1$ in it.
Orient the cap in such a way that the outer normal at $x$ points in the positive direction of the y-axis.
The intersection with $\partial K$ has minimal $y$-coordinate in case the random point in $\Delta_0$ lies at the bottom corner, nearest to $\Delta_1$, and the random point in $\Delta_1$ is in the point farthest from the boundary, see Figure \ref{abra:simulokorok}.
Let us denote these points by $a_0$ and $a_1$, respectively.

In the case when $K$ is a circle of radius $r$, we can exactly compute the $y$-coordinate of the intersection for a fixed $r$. 
The depth will be smaller than $c_3t$ for a suitable constant $c_3$ depending only on $K$. 
This $c_3$ is bounded as $r$ is strictly smaller than $1$.
Now we can specify the constant $c_2$ to be as large such that the  disc-caps $D(y_k,t_n)$ are far apart enough and the observed intersection point is not contained in any other disc-cap.
The constant $c_2$ depends only on $K$.
Therefore the statement of the Lemma is true for circles.

For a general convex disc $K$ with $C^2_+$ boundary, we estimate the $y$-coordinate of the intersection point as follows.
Consider the osculating circle of $K$ at the point $x$ with radius $(R_0(x)=)R_0<1$.
There exists an $\varepsilon>0$, such that in any neighbourhood of radius less than $\varepsilon$ of $x$, for the circles of radii $R_0+\varepsilon$ and $R_0-\varepsilon$ having the same tangent line as $K$ at $x$, it is true that $K$ is locally inside of the larger circle and the smaller circle is inside of $K$.

The line $H(x,t)$ meets these circles in $p_1,p_2$ and $q_1,q_2$, where the points with the same indices are close to each other.
Then for $i=1,2$, 
\begin{align}
	d(p_i,q_i)&=\sqrt{(2R_0+2\varepsilon-t)t}-\sqrt{(2R_0-2\varepsilon-t)t} \nonumber \\
	&=\sqrt{t}\left(\sqrt{2R_0+2\varepsilon-t}-\sqrt{2R_0-2\varepsilon-t}\right) \nonumber\\
	&=\sqrt{t}\left(\frac{4\varepsilon}{\sqrt{2R_0+2\varepsilon-t}+\sqrt{2R_0-2\varepsilon-t}}\right). \label{kor-tavolsag}
	\end{align}

Since $t\to0$, we may assume that $t<\varepsilon$, thus we can estimate \eqref{kor-tavolsag} from above as follows.
\begin{equation*}
\sqrt{t}\left(\frac{4\varepsilon}{\sqrt{2R_0+2\varepsilon-t}+\sqrt{2R_0-2\varepsilon-t}}\right)
\leq \frac{4}{\sqrt{2R_0}}\varepsilon\sqrt{t}\leq c_4\varepsilon\sqrt{t},
\end{equation*}
where $c_4=\max_{x\in\partial K}4/\sqrt{2R_0(x)}$.

\begin{figure}
	\centering
	\includegraphics[width=.8\textwidth]{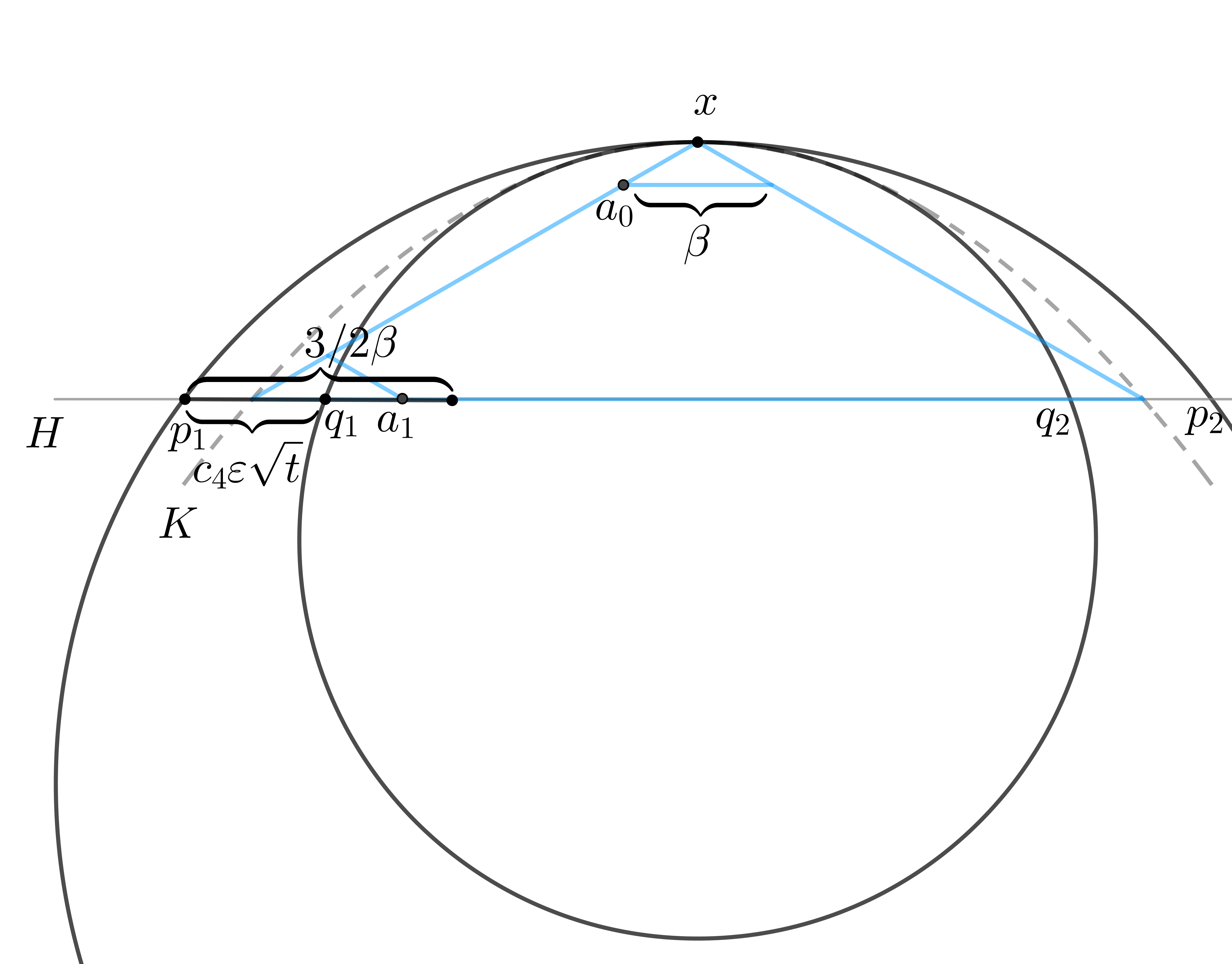}
	\caption{$K$ bounded between circles}\label{abra:simulokorok}
\end{figure}

Let $\beta$ denote the length of the side of $\Delta_0$ opposite to $x$.
Then the distance between the intersection of $\partial K$ with the line $H$ and $a_1$ is also $\beta$.
The order of magnitude of $\beta$ is $\sqrt{t}$.
We may assume that $\beta=c_5\sqrt{t}$ for some constant $c_5$.

Consider the unique point of $H\cap K$ which is at distance $3/2\beta$ from $p_1$.
The circular arc of radius $1$ determined by this point and $a_0$, meets the outer circle at a computable, bounded depth depending on $(R_0+\varepsilon)$.
In the case $c_4\cdot\varepsilon<c_5/2$, the circular arc incident with $a_0$ and $a_1$ meets the outer circle less deeply.
For this, we need that
\begin{equation*}
	\varepsilon<\frac{c_5}{2c_4}.
\end{equation*}

Note that constant on the right-hand side of the inequality depends only on $K$.
Since $t\to0$, $\varepsilon$ can be chosen smaller than that. The only restriction on $\varepsilon$ is that $K$ has to be locally between the two circles, that is, the intersection point $c_6t$ deep is in this range.
Therefore, we can choose $\varepsilon$ arbitrary small.


Thus for every $K$, the circular arc determined by the points $a_0$ and $a_1$ meets the boundary of $K$ in a depth of order $t$.
Therefore this arc will not intersect any other disc-cap $D(y_j,t_n)$.
\end{proof}

Let $\mathcal{F}$ denote the $\sigma$-algebra generated by the events $E_j$.
Consider the conditional variance on $\mathcal{F}$.
By the Law of Total Variance,
\begin{align}
	\Var A(K_n)&=\Ex\Var(A(K_n)\mid\F)+\Var\Ex(A(K_n)\mid\F)\nonumber\\
	&\geq \Ex\Var(A(K_n)\mid\F). \label{total-var}
\end{align}
	
For the set of indices $J\subseteq\left[m\right]$, let $E(J)$ be the event that the event $E_j$ occurs for exactly the indices $j\in J$, and does not occur for the other indices in $[m]$.
Then the conditional variance on $\mathcal{F}$ can be expanded in the following form:
\begin{equation}\label{osszeg}
	\Var(A(K_n)\mid\F)=\sum_{J\subseteq\left[m\right]}\Var(A(K_n)\mid E(J))\cdot\indi(E(J)).
\end{equation}

By Lemma \ref{fuggetlen}, in case the event $E(J)$ occurs, the area of $K_n$ can be written as the sum of $\hat{A}(Z_j)$ and the area of the hull of the remaining points:
\begin{equation}\label{osszeg-var}
    \Var(A(K_n)\mid E(J))=\Var\left[ A\left(\conv_r\left(X_n\setminus\{Z_1,\dots Z_j\}\right)\right)+\sum_{j\in J}\hat{A}(Z_j)\mid E(J)\right].
\end{equation}
Let the points of $X_n\setminus\{Z_1,\dots Z_j\}$ be fixed.
Then the first term on the right-hand side of \eqref{osszeg-var}, the area of the $r$-hull of the remaining points is constant, thus we can omit it.
The other terms are independent since the contributions of the points $Z_j$ are geometrically independent, therefore we can take the variance term-by-term.

This holds for every fixed set of points $X_n\setminus\{Z_1,\dots Z_j\}$, thus we have the following inequality. The variance $\Var_{Z_j}$ means that the variance is taken only for the corresponding variable.
\begin{equation*}
    \Var(A(K_n)\mid E(J))\geq \sum_{j\in J}\Var_{Z_j} \left(\hat{A}(Z_j)\mid E_j\right),
\end{equation*}

Substituting it in \eqref{osszeg}, we obtain that
\begin{equation}\label{dupla-szumma}
\sum_{J\subseteq\left[m\right]}\Var(A(K_n)\mid E(J))\cdot\indi(E(J))\geq
\sum_{J\subseteq\left[m\right]}\indi(E(J))\sum_{j\in J}\Var_{Z_j}\left(\hat{A}(Z_j)\mid E_j\right).
\end{equation}
Rearrange the right-hand side of \eqref{dupla-szumma} such that the sum is according to the index of $\Var_{Z_j}\left(\hat{A}(Z_j)\mid E_j\right)$:
\begin{align*}
	\sum_{J\subseteq\left[m\right]}\indi(E(J))\sum_{j\in J}\Var_{Z_j}\left(\hat{A}(Z_j)\mid E_j\right) &=\sum_{j=1}^{m}\Var_{Z_j}\left(\hat{A}(Z_j)\mid E_j\right)\sum_{\{J:j\in J\}}\indi(E(J))=\\
	&=\sum_{j=1}^{m}\Var_{Z_j}\left(\hat{A}(Z_j)\mid E_j\right)\indi(E_j).
\end{align*}

Thus, we have	
\begin{equation}
	\Var(A(K_n)\mid\F)\geq\sum_{j=1}^{m}\Var_{Z_j}\left(\hat{A}(Z_j)\mid E_j\right)\indi(E_j).
	\label{var-szetbontasa}
\end{equation}
	
By Lemma \ref{sapka-szorasok} together with \eqref{tn-ek}-\eqref{expectation}, \eqref{total-var}, and \eqref{var-szetbontasa},
\begin{align*}
	\Var A(K_n)&\geq\Ex\left(\sum_{j=1}^{m}\Var_{Z_j}\left(\hat{A}(Z_j)\mid E_j\right)\indi(E_j)\right)
	\gg\Ex\left(\sum_{j=1}^{m}t_n^3\indi(E_j)\right)\gg\\
	&\gg n^{-2}\left(\Ex\sum_{j=1}^m\indi(E_j)\right)
	\gg n^{-2}m\gg n^{-\frac53},
\end{align*}
which is \eqref{area-variance-a} of Theorem \ref{thm:also becsles}, thus we have finished the proof.

\section{The variance of the number of vertices}\label{proof-vertex}

We give the outline of the necessary changes in the previous argument to prove \eqref{vertex-variance-a} of Theorem \ref{thm:also becsles}, the lower bound of the variance of the number of vertices, see also Reitzner \cite{R05b}.

Let the disc-caps $D(y_j,t_n)$ be defined as before, and also the triangles $\Delta(y_n,t_n)$ together with the small triangles $\Delta_i(y_n,t_n)$, where $j\in\left[m\right]$ and $i\in\{0,1,2\}$.
Let $F_j$ denote the event that $\Delta_1(y_n,t_n)$ and $\Delta_2(y_n,t_n)$ each contain exactly one of the random points $x_1,\dots,x_n$, $\Delta_0(y_n,t_n)$ contains exactly two of $x_1,\dots,x_n$, and there is no other random point in the disc-cap $D(y_j,c_2t_n)$.
The probability of $F_j$ is
\begin{equation*}\label{fvalseg}
	\pr(F_j)\gg\binom{n}{4}\left(\frac1n\right)^4\left(1-\frac1n\right)^{n-4}\gg1,
\end{equation*}
thus, similar to \eqref{expectation}, we have
\begin{equation}\label{expectf}
	\Ex\left(\sum_{j=1}^{m}\indi(F_j)\right)\gg m.
\end{equation}

If the event $F_j$ occurs, denote the random points in $\Delta_0(y_j,t_n)$ by $Y_j$ and $Z_j$.
It follows from the proof of Lemma \ref{fuggetlen}, that in this case each one of the triangles $\Delta_i(y_n,t_n)$ $(i\in\{0,1,2\})$ contains a vertex of $K_n$.
Either only one of the points $Y_j$ and $Z_j$ is a vertex of the $K_n$, or both of them are.
Therefore the disc-cap $D(y_j,t_n)$ contains either $3$ or $4$ vertices, both of these events have positive probability.
It follows, taking the variance for only these two points, that
\begin{equation}\label{csszoras}
    \Var_{Y_j,Z_j}\left(\hat{f_0}(\left[Y_j,Z_j,z_{j,1},z_{j,2}\right])\mid F_j\right)\gg1,
\end{equation}
where $\hat{f_0}(\left[Y_j,Z_j,z_{j,1},z_{j,2}\right]$ denotes the number of vertices contained in the $j$-th disc-cap.

Let $\mathcal{G}$ be the $\sigma$-algebra generated by the events $F_j$.
Similar to the proof of Lemma \ref{fuggetlen}, it holds here as well, that the number of vertices in one of the disc-caps does not affect how many vertices are there in an other disc-cap $D(y_j, t_n)$.
Using this fact, together with \eqref{m},\eqref{total-var} and \eqref{expectf}-\eqref{csszoras}, we get
\begin{align*}
    \Var f_0(K_n)&\geq \Ex \Var(f_0(K_n)\mid \mathcal{G})\\
    &\geq\Ex\left(\sum_{j=1}^m\Var_{Y_j,Z_j}\left(\hat{f_0}(\left[Y_j,Z_j,z_{j,1},z_{j,2}\right])\mid F_j\right)\indi(F_j)\right)\\
    &\gg\Ex\left(\sum_{j=1}^m\indi(F_j)\right)\gg m\gg n^{\frac13},
\end{align*}
which is the statement of \eqref{vertex-variance-a} of Theorem \ref{thm:also becsles}, thus we have finished the proof.

\section{Acknowledgments}
The first author was partially supported by Hungarian NKFIH grant K134814. The third author was partially supported by  Hungarian NKFIH grant FK135392. 

This research was supported by grant NKFIH-1279-2/2020 of the Ministry for Innovation and Technology, Hungary.

\end{document}